\documentclass[11pt,oneside,article]{memoir}

\usepackage{tikz}
\usetikzlibrary{decorations.markings, decorations.pathmorphing, arrows.meta, calc, fit, quotes, cd, positioning, backgrounds, shapes}

\usepackage{pgfplots}
\pgfplotsset{compat = newest}

\usepackage{todonotes}
\usepackage{amsthm,amsmath,amssymb,amsfonts,mathtools}
\usepackage{paralist}
\usepackage{newpxtext}
\usepackage{newpxmath}
\usepackage{hyperref}
\hypersetup{colorlinks,linkcolor={blue!50!black},citecolor={red!50!black},urlcolor={blue!80!black}}\usepackage{amsmath}

\usepackage{algorithm}
\usepackage[noend]{algpseudocode}
\usepackage{algorithmicx}
\newcommand*\Let[2]{\State #1 $\coloneqq$ #2}
\algrenewcommand\alglinenumber[1]{
    {\textsf\footnotesize#1}}
\algrenewcommand\algorithmicrequire{\textbf{Precondition:}}
\algrenewcommand\algorithmicensure{\textbf{Postcondition:}}

\usepackage[backend=biber,style = alphabetic]{biblatex}
\DeclareMathOperator{\Tr}{Tr}

\newcommand{\NN}{\mathbb{N}}
\newcommand{\Real}{\mathbb{R}}
\newcommand{\BB}{\mathbb{B}}

\renewcommand{\ss}{\subseteq}
\newcommand{\ord}[1]{[#1]}
\newcommand{\From}[1]{\xleftarrow{#1}}
\newcommand{\To}[1]{\xrightarrow{#1}}
\newcommand{\ol}[1]{\overline{#1}}

\newcommand{\Arr}{\mathsf{Arr}}
\newcommand{\Entr}{\mathsf{Entr}}
\newcommand{\BBox}{\mathsf{BBox}}
\newcommand{\Plot}{\mathrm{Plot}}

\newcommand{\Pixel}{\mathrm{Pixel}}
\newcommand{\Cost}{\mathrm{Cost}^1}
\newcommand{\PCost}{\mathrm{Cost}^\infty}
\newcommand{\Clust}{\mathrm{Clust}}
\newcommand{\Bool}{\BB\mathrm{ool}}
\newcommand{\id}{\mathrm{id}}

\newcommand{\Rel}{\mathsf{Rel}}

\theoremstyle{plain}
\newtheorem{theorem}{Theorem}[section]
\newtheorem*{theorem*}{Theorem}
\newtheorem{proposition}[theorem]{Proposition}
\newtheorem{corollary}[theorem]{Corollary}
\newtheorem*{corollary*}{Corollary}

\newtheorem*{lemma*}{Lemma}

\theoremstyle{definition}
\newtheorem{definition}[theorem]{Definition}

\theoremstyle{remark}
\newtheorem{example}[theorem]{Example}
\newtheorem{remark}[theorem]{Remark}

\tikzset{
	pack/.style = 
		{circle, draw, minimum size = 4 mm, font = \normalsize}, 
	wd/.style = {
		draw, ellipse, minimum width = 100 mm, minimum height = 50 mm, scale = 1.0}, 
	invertwd/.style = 
		{draw, ellipse, minimum width = 50 mm, minimum height = 100 mm, scale = 1.0}, 
	helper/.style = 
		{draw, circle, minimum size = .01 mm, inner sep = 0 pt},
	epack/.style = 
	{draw, ellipse, minimum width = 25 mm, minimum height = 12.5 mm, font = \normalsize}
	}
	
	\tikzset{
   oriented WD/.style={
      every to/.style={out=0,in=180,draw},
      label/.style={
         font=\everymath\expandafter{\the\everymath\scriptstyle},
         inner sep=0pt,
         node distance=2pt and -2pt},
      semithick,
      node distance=1 and 1,
      decoration={markings, mark=at position .5 with {\arrow{stealth};}},
      ar/.style={postaction={decorate}},
      execute at begin picture={\tikzset{
         x=\bbx, y=\bby,
         every fit/.style={inner xsep=\bbx, inner ysep=\bby}}}
      },
   bbx/.store in=\bbx,
   bbx = 1.5cm,
   bby/.store in=\bby,
   bby = 1.75ex,
   bb port sep/.store in=\bbportsep,
   bb port sep=2,
   bb port length/.store in=\bbportlen,
   bb port length=4pt,
   bb min width/.store in=\bbminwidth,
   bb min width=1cm,
   bb rounded corners/.store in=\bbcorners,
   bb rounded corners=2pt,
   bb small/.style={bb port sep=1, bb port length=2.5pt, bbx=.4cm, bb min width=.4cm, bby=.7ex},
   bb/.code 2 args={
      \pgfmathsetlengthmacro{\bbheight}{\bbportsep * (max(#1,#2)+1) * \bby}
      \pgfkeysalso{draw,minimum height=\bbheight,minimum width=\bbminwidth,outer sep=0pt,
         rounded corners=\bbcorners,thick,
         prefix after command={\pgfextra{\let\fixname\tikzlastnode}},
         append after command={\pgfextra{\draw
            \ifnum #1=0{} \else foreach \i in {1,...,#1} {
               ($(\fixname.north west)!{\i/(#1+1)}!(\fixname.south west)$) +(-\bbportlen,0) coordinate (\fixname_in\i) -- +(\bbportlen,0) coordinate (\fixname_in\i')}\fi 
            \ifnum #2=0{} \else foreach \i in {1,...,#2} {
               ($(\fixname.north east)!{\i/(#2+1)}!(\fixname.south east)$) +(-\bbportlen,0) coordinate (\fixname_out\i') -- +(\bbportlen,0) coordinate (\fixname_out\i)}\fi;
         }}}
   },
   bb name/.style={append after command={\pgfextra{\node[anchor=north] at (\fixname.north) {#1};}}}
}

\settrims{0pt}{0pt} 
\settypeblocksize{*}{36pc}{*} 
\setlrmargins{*}{*}{1} 
\setulmarginsandblock{.9in}{.9in}{*} 
\setheadfoot{\onelineskip}{2\onelineskip} 
\setheaderspaces{*}{1.5\onelineskip}{*} 
\checkandfixthelayout

\makeatletter
\def\blfootnote{\gdef\@thefnmark{}\@footnotetext}
\makeatother

\setsecheadstyle{\bfseries\large\raggedright}
\setsubsecheadstyle{\bfseries\raggedright}

\title{Pixel Arrays: A fast and elementary method\\for solving nonlinear systems}
\author{David I. Spivak%
\thanks{Spivak was supported by AFOSR grants FA9550--14--1--0031 and FA9550-17-1-0058, and NASA grant NNH13ZEA001N-SSAT.}
\;\; Magdalen R. C. Dobson\\ Sapna Kumari\;\; Lawrence Wu
}
\date{\vspace{-.25in}}
\begin{document}
\firmlists*
	
\maketitle

\pagestyle{companion}

\begin{abstract}\label{abstract}
We present a new method, called the pixel array method, for approximating all solutions in a bounding box for an arbitrary nonlinear system of relations. In contrast with other solvers, our approach requires that the user must specify which variables are to be exposed, and which are to be left latent. The entire solution set is then obtained---in terms of these exposed variables---by performing a series of array multiplications on the $n_i$-dimensional plots of the individual relations $R_i$. This procedure introduces no false negatives and is much faster than Newton-based solvers. The key is the unexposed variables, which Newton methods can make no use of. In fact, we found that with even a single unexposed variable our method was more than 10x faster than Julia's NLsolve. Due to its relative simplicity, the pixel array method is also applicable to a broader class of systems than Newton-based solvers are. The purpose of this article is to give an account of this new method.\\

\end{abstract}

\chapter{Introduction} \label{ch:introduction}

The need to compute solutions to systems of equations or inequalities is ubiquitous throughout mathematics, science, and engineering. A great deal of work is continually spent on improving the efficiency of linear systems solvers \autocite{Coppersmith.Winograd:1987a,Fatahalian.Sugerman.Hanrahan:2004a,Goto.Geijn:2008a,Dalton.Olson.Bell:2015a} and algebro-geometric approaches have also been developed for solving systems of polynomial equations \autocite{Grigorev.Vorobjov:1988a,Canny.Kaltofen.Yagati:1989a,Sturmfels:2002a}. Less is known for systems of arbitrary continuous functions, and still less for systems involving inequalities or other relations \autocite{Broyden:1965a,Martinez2000}. Techniques for solving nonlinear systems are often highly technical and specific to the particular types of equations being solved. According to \cite{Martinez2000}, "all practical algorithms for solving [nonlinear systems] are iterative"; they are designed to find one solution near a good initial guess. 

We present a new method with which to find an approximation to the entire solution set---in a bounding box, and for a user-specified subset of "exposed" variables---of an arbitrary system of relations. This approach has the following features:
\begin{compactitem}
	\item it returns \emph{all solutions in a given bounding box};
	\item it introduces \emph{no false negatives};
	\item it works for \emph{non-differentiable} or even \emph{discontinuous} functions;
	\item it is \emph{not iterative} and requires no initial guess, in contrast with quasi-Newton methods;
	\item it is \emph{elementary} in the sense that it relies only on generalized matrix arithmetic; and
	\item it is \emph{faster} than quasi-Newton methods for this application.
\end{compactitem}
The "application" to which we refer just above is that in which the user only requires solution values for certain variables. The user chooses any subset of variables, and solutions to the system will be reported only in terms of the chosen or "exposed" variables. Solution values of the other "unexposed" variables cannot be recovered, except by exposing them in the next run of the  algorithm. The existence of unexposed variables is key to the speed of our method and therefore to its advantage over quasi-Newton methods. We found that having even a single unexposed variable can amount to a 10x speedup over Julia's NLsolve; see Section~\ref{ch:Results}.

We call our technique the \emph{pixel array} (PA) method. While it has many advantages, the PA method also has limitations. One such limitation, as discussed above, is that the solution values for unexposed variables are lost.

 
\section{A simple example}\label{sec: Motivating Example}
Here we give an overly-simplified example to fix ideas.

Suppose we plot the two equations $x^2=w$ and $w=1-y^2$ as graphs on a computer screen. The result for each equation, say the first one above, will be a multidimensional array of pixels---some on and some off---that represents the set of $(x,w)$-points which satisfy the equation. Thus we can plot each equation as a \emph{matrix of booleans}---True's and False's---representing its graph; say $M$ for the first equation and $N$ for the second. What happens if we multiply the two matrices together? It turns out that the resulting matrix $MN$ represents the set of $(x,y)$-pairs for which there exists a $w$ simultaneously satisfying both equations in the system. In other words, ordinary matrix multiplication returns the plot of a circle, $x^2+y^2=1$; see Figure~\ref{fig: motivating example} and sample code in Appendix~\ref{sec:sample_code}.
\begin{figure}
	\includegraphics[width=\textwidth]{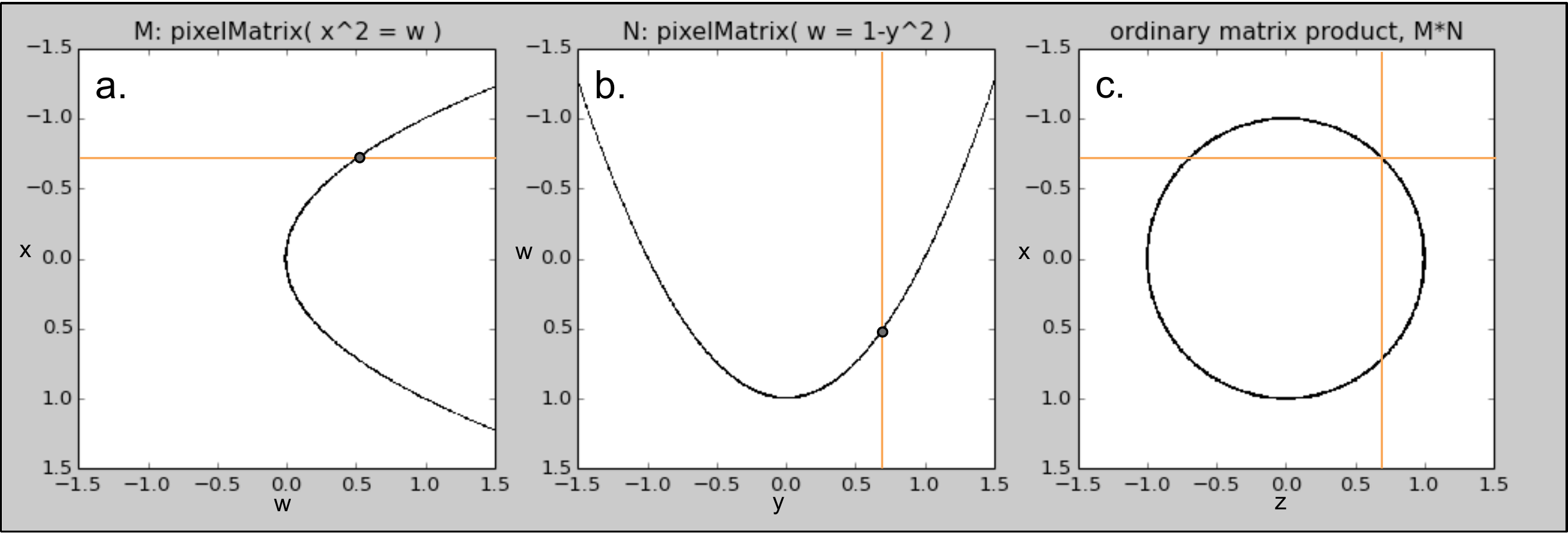}
	\caption{Parts a and b show plots of $x^2=w$ and $w=1-y^2$, as $100\times 100$ pixel matrices, where each entry is a 0 or 1, shown as white and black points, respectively. The graphs appear rotated $90^\circ$ clockwise, in order to agree with matrix-style indexing, where the first coordinate is indexed downward and the second coordinate is indexed rightward. The matrices from parts a and b are multiplied, and the result is shown in part c. The horizontal and vertical lines in parts a and b respectively indicate a sample row and column whose dot product is 1, hence the pixel is on at their intersection point in part c. The fact that the matrix product looks like a circle is not a coincidence; it is the graph of the simultaneous solution to the system given in parts a and b, which in this case can be rewritten to a single equation $x^2=1-y^2$.}
	\label{fig: motivating example}
\end{figure}

A couple of observations are in order for the simple system above. First, one can see immediately by looking at Figure~\ref{fig: motivating example} that we are not in any way intersecting two plots; matrix multiplication is a very different operation. Second, despite the simplicity of the system given here---one may simply eliminate the variable $w$ directly\footnote{It turns out that even though one may simply eliminate $w$ and plot $x^2+y^2=1$ directly, it is in fact faster to plot the functions separately and combine the results using matrix multiplication, basically because plotting functions is faster than plotting relations.
}---the PA method works in full generality, as we will see in Section~\ref{sec:math_foundations}. 

Above we combined two relations using matrix multiplication $MN$, but for systems with multiple relations we need to use a more general array multiplication formula. The reason for this is that each relation can involve more than two variables, and there are many different ways that these variables may be shared between the relations. In fact, there is one array multiplication algorithm that simultaneously generalizes matrix multiplication, trace, and Kronecker (tensor) product; we will explain it in Section~\ref{sec:GAM}.

\paragraph{Our implementation} The plots and time-estimates presented in this paper were obtained using an implementation of the pixel array method in Julia \cite{Bezanson}, an easy-to-learn technical computing language with powerful array handling capabilities.

\section{Plan of the paper}
We begin in Section~\ref{sec:inputs_outputs_artifacts} by describing what the pixel array (PA) method does---i.e.\ what its inputs and outputs are---as well as briefly describing the mathematical by-products that are used along the way and then giving some examples. In Section~\ref{sec:math_foundations} we give details about the mathematical underpinnings of the PA method, including how to plot equations and combine them using the general array multiplication algorithm. We also discuss error bounds for this method and explain how a generalized associative law allows us to group or "cluster" our array multiplication process for efficiency.  In Chapter~\ref{ch:Results} we discuss our attempts to benchmark the PA method. 

\chapter{Pixel array method: its inputs, outputs, and intermediate structures}\label{sec:inputs_outputs_artifacts}

In this section we describe what the pixel array (PA) method does in terms of: what the user supplies as input, what intermediate structures are then created, and finally what is returned to the user as output. We then give some examples.

\section{Input to the pixel array method}\label{sec:input}

The input to the PA method is a set of relations (equations, inequalities, etc.), a discretization---i.e.\ a range and a resolution---for each variable, and a set of variables to solve for. That is, certain variables are considered \emph{unexposed} whereas others are  \emph{exposed}. Here is an exemplar of what can be input to the PA method:
\begin{alignat}{3}
&\text{Solve relations:}&&\qquad  x^2 + 3 |x-y|- 5 = 0\tag{$R_1$}\\
&							&&\qquad y^2v^3 - w^5 \le 0\tag{$R_2$}\\
&							&&\qquad \cos(u+zx) - w^2 = 0 \tag{$R_3$}\\\nonumber
&\text{Range and resolution:}&&\qquad u,v,x,y\in[-2,2)@50;\quad w,z\in[-1,1]@80\\\nonumber
&\text{Expose variables:}&&\qquad(v,z)
\end{alignat}
The first thing to extract from this setup is how variables are shared between relations. Relation $R_1$ has variables $x,y$; relation $R_2$ has variables $v,w,y$; relation $R_3$ has variables $u,w,x,z$; and variables $v,z$ are to be exposed. These facts together constitute the wiring diagram, which we discuss in Section~\ref{sec:intermediate}.

\section{Intermediate structures produced by the pixel array method}\label{sec:intermediate}
Running the PA method uses several sorts of mathematical structures in order to solve the system. The most important of these is probably what we call a \emph{wiring diagram}, which describes how the variables are shared between relations, as well as what variables are to be exposed. 

\subsection{Wiring diagram}
For the system of three equations shown above in Section~\ref{sec:input}, the wiring diagram $\Phi$ would look like this:
\begin{equation}\label{ex:wd}
\begin{tikzpicture}[baseline=(wd)]
	
	\node[pack] (P1) {$P_1$};
	\node[pack, below left=.5 and 1 of P1] (P2) {$P_2$};
	\node[pack, below right=.5 and 1 of P1] (P3) {$P_3$};
	
	\node[helper,inner sep = .8 pt,fill=black,below left=.1 and .1 of P3] (H3) {};
	\node[wd, scale = .53, fit={($(P1.north)+(0,-5pt)$) (P2) (P3) (H3)}] (wd) {}; 

	\draw (P2) --+(-1.2,0);
	\draw (P1) to (P2);
	\draw (P2) to (P3);
	\draw (P3) -- +(1.2,0);
	\draw (P1) to (P3);
	\draw (P3) to (H3);
	
	\node at ($(P1)!.5!(P2) + (0,5pt)$) {$y$};
	\node at ($(P1)!.5!(P3) + (0,6pt)$) {$x$};
	\node at ($(P2)!.5!(P3) + (0,5pt)$) {$w$};

	\node at ($(P2)+(-.7,5pt)$) {$v$};
	\node at ($(P3)+(.7,5pt)$) {$z$};
	\node at ($(H3)+(-5pt,0)$) {$u$};
	
	\node at ($(wd.south)+(0,.3)$) {$P'$};
	\node[above=0 of wd] {$\Phi\colon P_1,P_2,P_3\to P'$};
\end{tikzpicture}
\end{equation}
We will explain how to represent a diagram like \eqref{ex:wd} set-theoretically in Section~\ref{sec:Packs_and_WDs}.

A wiring diagram $\Phi$ consists of a single outer circle and several inner circles. Each circle has a number of ports, and each port refers to a single variable of the system. Each variable has a \emph{discretization}, namely its range and resolution $r\geq 2$, and the resolution is attached to the port. One can see in \eqref{ex:wd} that $P_1$ has ports $x$ and $y$, and it was specified above that the resolution of $y$ is $50$. We call a circle with its collection of ports a \emph{pack}. For each variable in the system there is a \emph{link} which connects all ports referring to that variable. In \eqref{ex:wd} all links other than $u$ connect two ports, e.g.\ the link for $y$ connects a port of $P_1$ to a port of $P_2$.

Thus the wiring diagram pulls all this information together: it consists of several inner packs wired together inside an outer pack. The wiring diagram \eqref{ex:wd} was denoted $\Phi\colon P_1,P_2,P_3\to P'$ because it includes three inner packs $P_1,P_2,P_3$ and one outer pack $P'$. Thus the first intermediate structure created by the pixel array method is a wiring diagram; see Definition~\ref{def:wd}.

A wiring diagram in turn specifies an \emph{array multiplication formula}, which is roughly a generalization of matrix multiplication to arrays of arbitrary dimension. If given a plot of each relation $R_i$ in the system, the array multiplication formula specifies how these plots should be combined to produce a plot of the solution to the whole system. Before discussing this further, we must explain what plots are in a bit more detail.

\subsection{Plots}
The plot of a relation $R$ is a boolean array that discretizes $R$ with the specified range and resolution. For example pack $P_2$ includes ports $v,w,y$ with resolutions $50,80,50$; thus it will be inhabited by a $3$-dimensional array of size $50\times80\times 50$. The entries of this array may be called \emph{pixels}; each pixel refers to a sub-cube of the range $[-1,1]\times[-2,2]\times[-1,1]$. The pixels are adjacent sub-cubes of size $\frac{2}{50}\times\frac{4}{80}\times\frac{2}{50}$ in this case. The value of each pixel is boolean---either on or off, $1$ or $0$, True or False---depending on whether the relation holds somewhere in that sub-cube, or not.%
\footnote{Above we define the entries in our arrays to be booleans, but in fact one can use values in any semi-ring, such as $\NN$ or $\Real_{\geq0}$, and everything in this article will work analogously. Rather than indicating \emph{existence} of solutions in each pixel, other semi-rings allow us to indicate \emph{densities} of solutions in each pixel.}
We denote the plot of relation $R$ inside pack $P$ by $\Plot_P(R)$.

The initial plots may be obtained by evaluation at each pixel.
These plots may be considered the second sort of intermediate structure produced by the PA method.

\subsection{Array multiplication formula}
We now return to our brief description of the array multiplication formula, which is derived immediately from the wiring diagram. The basic idea is that whereas a matrix has two dimensions---rows and columns---an array can have $n$ dimensions for any $n\in\NN$. Thus, whereas a matrix can be multiplied on the left or right, arrays can be multiplied in more ways. There is a more elaborate sense in which array multiplication is associative, as we describe in Section~\ref{sec:GAM}. 

Recall that a pack $P$ includes a number of ports, each labeled by its range and resolution. For each port $p\in P$, we denote its resolution by $r(p)$; in fact, we assume that $r(p)\geq 2$ because otherwise that port adds no information. The set of resolutions for $P$ determines the size---the set of dimensions---of an array; let $\Arr(P)$ denote the set of all arrays having this size and boolean entries. Thus $\Arr(P_3)$ is the set of all possible $50\times80\times80$ arrays, and the plot of $R_3$ is one of them, which we denote $p_3\coloneqq\Plot_{P_3}(R_3)\in\Arr(P_3)$.

The wiring diagram $\Phi\colon P_1,P_2,P_3\to P'$ specifies an array multiplication formula, which is a function 
\[\Arr(\Phi)\colon\Arr(P_1)\times\Arr(P_2)\times\Arr(P_3)\to\Arr(P').\]
Thus, given a plots $p_1,p_2,p_3$, the array multiplication formula for wiring diagram $\Phi$ returns an output plot $\Arr(\Phi)(p_1,p_2,p_3)\in\Arr(P')$.

\subsection{Clustering}
The array multiplication formula allows one to patch together these individual plots into a solution to the entire system. This can be done all at once, for example given matrices $A,B,C$, generalized array multiplication allows one to multiply them all at once $ABC$ rather than associatively as $A(BC)$ or $(AB)C$. 

It turns out that multiplying many arrays together all at once, at least using the naive algorithm, is often inefficient. However, the associativity of array multiplication (see Section~\ref{sec:GAM}) allows us to solve the system more efficiently by clustering the wiring diagram. The most efficient cluster tree may be difficult to ascertain, but even modest clustering is useful. We will discuss clustering in more detail in Section~\ref{sec:clustering}. A choice of cluster tree for the wiring diagram, and hence strategy for multiplying the arrays, is the last intermediate structure required by the pixel array method.

\section{Output of the pixel array method}
Once the choice of clustering has been made, the plots are combined accordingly. Regardless of the clustering, the output is a plot whose dimensions are given by the discretizations of the exposed variables. 

In the particular example above, the exposed variables are the ports of pack $P'$, so the output will be an element of $\Arr(P')$, i.e.\ a $50\times80$ matrix of booleans. Its entries correspond to those $(v,z)$-pairs for which a simultaneous solution to $R_1,R_2,R_3$ exists. Above, we denoted this 2-dimensional array by $\Arr(\Phi)(p_1,p_2,p_3)$. 

\section{A couple more examples}\label{sec: examples from implementation}
In this section we give two examples from our Julia implementation which illustrate some features of the pixel array method.

\begin{example}\label{ex: butterfly}
Suppose one is asked for all $(w,z)$ pairs for which the following system of equations has a solution:%
\footnote{One may wonder whether the $10^{-5}z$, the $10^{-3}x$, etc.\ in fact make a perceptible difference in the resulting plot. We leave it to the reader to experiment for him- or herself.}
\begin{alignat}{3}
&\text{Solve relations:}&&\qquad  \cos\left(\ln(z^2 + 10^{-3}x)\right) - x +10^{-5}z^{-1} = 0\tag{Equation 1}\\
&							&&\qquad \cosh(w+10^{-3}y) + y + 10^{-4}w = 2\tag{Equation 2}\\
&							&&\qquad \tan(x+y)(x-2)^{-1}(x+3)^{-1}y^{-2}=1 \tag{Equation 3}\\\nonumber
&\text{Range and resolution:}&&\qquad w,x,y,z\in[-3,3)@125\\\nonumber
&\text{Expose variables:}&&\qquad(w,z)
\end{alignat}
The answer can be obtained by matrix multiplication; see the graph labeled 'Result' in Figure~\ref{fig:butterfly}.
Note that there are points at which Equation 3 is undefined; at such points our plotting function simply refrains from turning on the corresponding pixels, and the array multiplication algorithm proceeds as usual.
\begin{figure}
\includegraphics[width=\textwidth]{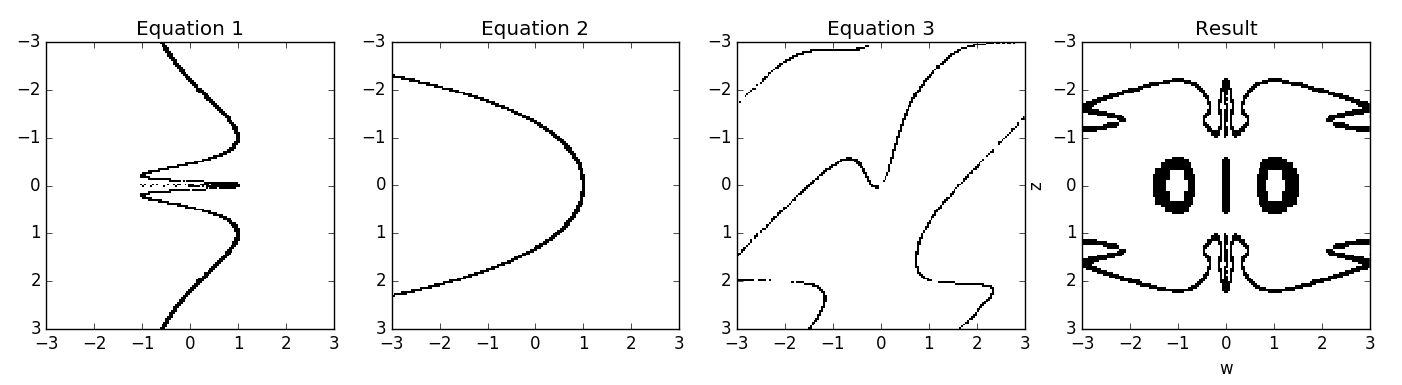}
\caption{Visualizing the entire solution set on exposed variables can reveal patterns.}
\label{fig:butterfly}
\end{figure}

\end{example}

\begin{example}\label{ex:3d_arrays}
Suppose given the input to the left in Figure~\ref{fig:regularity}. Each input plot is 3-dimensional, so we do not show them here. The result of the system is shown to the right.
\begin{figure}[h]
\begin{center}
\parbox{1.2in}{
\begin{alignat*}{3}
&\text{Solve relations:}&&\qquad  \tan(y+w)+\exp(x)=2\\
&						&&\qquad x^3+\cos(\ln(y^2)) = 1.5v\\
&							&&\qquad w+z+10^{-1}v = 0.5\\
&\text{Range and resolution:}&&\qquad v,x,z\in[-3,3)@75\\
&						&&\qquad w,y\in[-2.5,2.5]@75\\
&\text{Expose variables:}&&\qquad(w,y)
\end{alignat*}
}
\hspace{.5in}
\parbox{1.8in}{
\includegraphics[width= 1.8in]{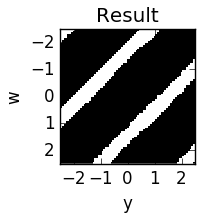}
}
\end{center}
\caption{Regularity in a more complex system.}
\label{fig:regularity} 
\end{figure}
The fact that the result is "solid"---i.e.\ 2-dimensional---is not surprising, because the system includes three equations and five unknowns. 
\end{example}

\chapter{Mathematical foundations}\label{sec:math_foundations}

In this section we explain the mathematical underpinnings of the pixel array method.

\addtocounter{section}{-1}
\section{Notation}\label{notation}
For a finite set $S$, we write $\#S\in\NN$ for its cardinality, and for a number $n\in\NN$, we will use brackets to denote the finite set
\[\ord{n}\coloneqq\{1,\ldots,n\}.\] 
For sets $A$ and $B$, we denote their Cartesian product by $A\times B$, and a multi-fold Cartesian product is denoted using the symbol $\prod$. We denote the disjoint union of sets by $A\sqcup B$. We denote functions from $A$ to $B$ by $A\to B$, injective functions by $A\hookrightarrow B$, and surjective functions by $A\twoheadrightarrow B$.

Throughout this article, we consider arrays whose values are booleans---elements of $\Bool=\{0,1\}$---which form a semiring, in the sense that there is a notion of $0,1,+,*$ and these are associative, distributive, unital, etc. For $\Bool$ the operations $+$ and $*$ are given by OR and AND ($\vee$ and $\wedge$). Moreover, $\Bool$ is a partially ordered semiring, meaning it has an ordering ($0\leq 1$) that are appropriately preserved by $+$ and $*$; see \cite{Golan2003}. In fact, all the ideas in this article work when $\Bool$ is replaced by an arbitrary partially ordered semiring $\BB$. We will generally suppress this fact to keep the exposition readable by a broader audience.

\section{Packs and wiring diagrams}\label{sec:Packs_and_WDs}

A wiring diagram can be visualized as a bunch of inner packs---circles with ports---wired together inside an outer pack. Each port is labeled by its range $[a,b)$ and resolution $r\geq 2$, and two ports can be wired together only if they have the same resolution.

\begin{definition}\label{def:pack}
A \emph{pack} is a tuple $(P,a,b,r)$, where $P$ is a finite set, $a,b\colon P\to\Real$ are functions with $a(p)\leq b(p)$ for all $p\in P$, and $r\colon P\to\NN_{\geq 2}$ is a function. Each element $p\in P$ is called a \emph{port}, $a(p)$ and $b(p)$ are called its \emph{bounds} and the half-open interval $[a(p),b(p))\ss\Real$ is called its \emph{range}, and $r(p)\geq 2$ is called its \emph{resolution}. 

We define the set of \emph{entries in $P$} to be the following product of finite sets:
\begin{equation}\label{eqn:def_Entr}
\Entr(P)\coloneqq\prod_{p\in P}\ord{r(p)}.
\end{equation}
And we define the \emph{bounding box for $P$} to be the product of closed intervals
\begin{equation}\label{eqn:def_BBox}
\BBox(P)\coloneqq\prod_{p\in P}[a_p,b_p)\ss\Real^{\#P}.
\end{equation}
We sometimes denote a pack $(P,a,b,r)$ simply by $P$, for typographical reasons. 
\end{definition}

Let $P_1,\ldots,P_n,P'$ be packs, and assume they are disjoint sets for convenience. A wiring diagram with inner packs $P_1,\ldots,P_n$ and outer pack $P'$, denoted $\Phi\colon P_1,\ldots,P_n\to P'$, can be defined as a surjective function $\varphi$ from the (disjoint) union $P_1\sqcup\cdots\sqcup P_n$ of all the inner ports onto a set $\Lambda$ of \emph{links}, and a function associating a link to each outer port. If two ports map to the same link, we say that they are \emph{linked}. We require that if two ports are linked then they must have the same bounds and resolution. We can package all this as follows.

\begin{definition}\label{def:wd}
Let $P_1,\ldots,P_n,P'$ be packs, let $\ol{P}\coloneqq P_1\sqcup\cdots\sqcup P_n\sqcup P'$ be their disjoint union with induced bounds $a,b\colon\ol{P}\to\Real$ and resolution function $r\colon\ol{P}\to\NN_{\geq2}$. A \emph{wiring diagram} $\Phi\colon P_1,\ldots,P_n\to P'$ consists of a tuple $\Phi=(\Lambda,\varphi,a_\Lambda,b_\Lambda,r_\Lambda)$ where 
\begin{compactitem}
  \item $\Lambda$ is a finite set, 
  \item $\varphi\colon\ol{P}\twoheadrightarrow \Lambda$ is a function such that the restriction $P_1\sqcup\cdots\sqcup P_n\to\Lambda$ is surjective,
  \item $a_\Lambda,b_\Lambda\colon\Lambda\to\Real$ are functions such that $a_\Lambda\circ\varphi=a$ and $b_\Lambda\circ\varphi=b$, and
  \item $r_\Lambda\colon\Lambda\to\NN_{\geq2}$ is a function such that $r_\Lambda\circ\varphi=r$.
\end{compactitem}
Note that $\varphi$ being surjective implies that if $a_\Lambda$ exists, it is necessarily unique; similarly with $b_\Lambda$ and $r_\Lambda$.
\end{definition}

\begin{example}\label{ex:wiring_diagram}
Consider the wiring diagram $\Phi\colon P_1,P_2,P_3\to P'$ shown below:
\begin{equation}\label{ex:threefold}
\begin{tikzpicture}[baseline=(outer.center)]
	\node[wd, scale = .65, minimum width = 90 mm] at (0,0) (outer) {};
	\node[above=.3cm of outer.south] {$P'$};
	
	\node[pack] at (-1.4,    0) (P1) {$P_1$};
	\node[pack] at ( 1.3,  0.8) (P2) {$P_2$};
	\node[pack] at ( 1.3, -0.8) (P3) {$P_3$};
	
	\node[helper, inner sep = 1 pt, fill = black] at (   0,    0) (h1) {};
	\node[helper] at (-3.3,    0) (h2) {};
	\node[helper] at ( 3.1, -0.8) (h3) {};
	\node[helper] at (3.1, .8) (h5) {};
	
	\node[helper, inner sep = 1 pt, fill = black] at (0.3, 0.8) (h4) {};
	\draw (P1) to (h1);
	\draw (P2) to (h1);
	\draw (P3) to (h1);
	\draw (P1) to (h2);
	\draw (P3) to (h3);
	\draw (P2) to (h4);
	\draw (P2) to (h5);
	\draw (P3) to (P2);

	\node at (-2.4,  0.2) {$x$};
	\node at (-0.4,  0.2) {$y$};
	\node at ( 2.3, -0.6) {$z$};
	\node at ( 0.5,  1.0) {$w$};
	\node at ( 2.3, .6) {$v$};
	\node at (1.6,0) {$u$};
\end{tikzpicture}
\end{equation}
Let's assume every port $p$ has bounds $[a(p),b(p)]=[-1,1]$ and resolution $r(p)=20$. In the diagram, we see that pack $P_1$ has two ports $(x_1,y_1)$, pack $P_2$ has ports $(u_2,v_2,w_2,y_2)$, pack $P_3$ has ports $(u_3,y_3,z_3)$, and the outer pack $P'$ has ports $(x',v',z')$. Subscripts and primes' were added to make the sets disjoint. The wiring diagram includes six links $\Lambda=\{\ell_u,\ell_v,\ell_w,\ell_x,\ell_y,\ell_z\}$. The linking function $\varphi$ is given by
\[
u_2,u_3\mapsto\ell_u,\;\;\;
v_2,v'\mapsto\ell_v,\;\;\;
w_1\mapsto\ell_w,\;\;\;
x_1,x'\mapsto\ell_x,\;\;\;
y_1,y_2,y_3\mapsto\ell_y,\;\;\;
z_3,z'\mapsto\ell_z.
\]
The only possible $a_\Lambda,b_\Lambda,r_\Lambda$ are given by $a_\Lambda(\ell)=-1, b_\Lambda(\ell)=1, r_\Lambda(\ell)=20$ for all $\ell\in\Lambda$.
\end{example}

\begin{remark}
For any wiring diagram $\Phi$, the tuple $(\Lambda,a_\Lambda,b_\Lambda,r_\Lambda)$ of links and their bounds and resolutions in fact has the structure of a pack. It does not appear to be intuitively helpful to imagine the set of links as a pack; however, definitions like \eqref{eqn:def_Entr}, \eqref{eqn:def_BBox}, and Definition~\ref{def:entries} do apply.
\end{remark}

\begin{definition}\label{def:entries}
    Let $P=\{p_1,\ldots,p_n\}$ be a pack with resolution $r\colon P\to\NN_{\geq2}$, and let $\Entr(P)=\prod_{p \in P} [r(p)]$ be its set of entries as in \eqref{eqn:def_Entr}. A size-$P$ boolean array is a function $A\colon\Entr(P)\to\{0,1\}$. Given an array $A\in\Arr(P)$ and an entry $e\in\Entr(P)$, we refer to $A(e)\in\{0,1\}$ as the \emph{value of $A$ at $e$}.

We define $\Arr(P)$ to be the set of size-$P$ boolean arrays.%
\footnote{
Note that if $n=0$ then $\Entr(P)$ is an empty product, so a $0$-dimensional array consists of a single boolean entry, $\Arr(\emptyset)=\{0,1\}$. Also note that $\Arr(P)$ can be given the structure of a module over the semiring $\Bool$: arrays can be scaled or added. However, we continue to speak of $\Arr(P)$ as a set.} 
\end{definition}

\begin{example}
In Example~\ref{ex:wiring_diagram}, pack $P_3=\{u_3,y_3,z_3\}$, each with resolution $20$. Then $\Entr(P_3)$ has $20^3=8000$ elements, and $\Arr(P_3)$ is the set of $3$-dimensional arrays of size $20\times 20\times 20$.
\end{example}

Recall that a wiring diagram $\Phi\colon P_1,\ldots,P_n\to P'$, includes a function $\varphi\colon\ol{P}=P_1\sqcup\cdots\sqcup P_n\sqcup P'\to\Lambda$. Composing with the inclusion $P_i\to\ol{P}$, for any $1\leq i\leq n$, returns a function $\varphi_i\colon P_i\to\Lambda$ that preserves the resolutions. These functions allow us to project any array whose dimension is specified by $\Lambda$ to an array whose dimension is specified by $P_i$: every entry for $\Lambda$ projects to an entry for $P_i$. Thus we naturally obtain functions between sets of entries
\begin{equation}\label{eqn:Entr_map}
\Entr_\Phi^i\colon\Entr(\Lambda)\to\Entr(P_i)
\qquad\text{and}\qquad
\Entr_\Phi'\colon\Entr(\Lambda)\to\Entr(P').%
\footnote{
It may be helpful to visualize the sets of entries and the functions provided by $\Phi$ as follows:
\[
\begin{tikzcd}[row sep=-3pt,column sep=large,ampersand replacement=\&]
	\Entr(P_1)
	\\[-6pt]
	\vdots\&\Entr(\Lambda)\ar[lu,"\Entr_\Phi^1"']\ar[ld,"\Entr_\Phi^n"]\ar[r,"\Entr_\Phi'"]\&\Entr(P')
	\\
	\Entr(P_n)
\end{tikzcd}
\]
}
\end{equation}
We similarly obtain continuous functions between the bounding boxes:
\begin{equation}\label{eqn:BBox_map}
\BBox_\Phi^i\colon\BBox(\Lambda)\to\BBox(P_i)
\qquad\text{and}\qquad
\BBox_\Phi'\colon\BBox(\Lambda)\to\BBox(P')
\end{equation}
The surjectivity requirement in Definition~\ref{def:wd} implies that the induced maps to the product are injective, so we have:
\begin{equation}\label{eqn:injectivity}
\Entr(\Lambda)\hookrightarrow\prod_i^n\Entr(P_i)
\qquad\text{and}\qquad
\BBox(\Lambda)\hookrightarrow\prod_i^n\BBox(P_i).
\end{equation}

\begin{example}
    In Example~\ref{ex:wiring_diagram}, the resolution for each port (and hence link) was assigned to be 20; there were 6 links, so $\Entr(\Lambda)=[20]^6$. Applied to an entry $(c_u,c_v,c_w,c_x,c_y,c_z)\in\Entr(\Lambda)$, the various functions $\Entr_\Phi$ in \eqref{eqn:Entr_map} return the entries
\[
(c_w,c_y)\in\Entr(P_1),\;\; (c_v,c_x,c_y)\in\Entr(P_2),\;\; (c_u,c_w,c_x,c_z)\in\Entr(P_3),\;\; (c_v,c_z)\in\Entr(P').
\]
\end{example}

We will discuss generalized array multiplication in Section~\ref{sec:GAM}. We conclude this subsection with some remarks about other mathematical ways to think about wiring diagrams.
 
\begin{remark}\label{rem:pointed_hypergraph}

Every wiring diagram $\Phi\colon P_1,\ldots,P_n\to P'$ has an underlying hypergraph $H(\Phi)$ where vertices are packs $V=\{P_1,\ldots,P_n,P'\}$ and where edges (hyperedges) are the links $\Lambda$. We need to be a bit careful about what we mean by hypergraph, however. First we need to label each edge by a resolution $r\geq 2$, so $H$ is a \emph{weighted hypergraph}. Second, a wiring diagram contains a chosen vertex, namely its outer pack; so $H$ is also a \emph{pointed hypergraph}. Third, there can be nontrivial "loops" in the sense that an edge can link to the same vertex multiple times, and similarly two different edges can link the same set of vertices; so $H$ is a \emph{multi-hypergraph}. All together $\Phi$ has the structure of a weighted pointed multi-hypergraph; it is given by three functions
\[V\From{\pi}\ol{P}\To{\varphi}\Lambda\To{r}\NN_{\geq2},\]
where $\Phi=(\Lambda,r,\varphi)$ and $\ol{P}$ are as in Definition~\ref{def:wd}, and where $\pi$ is the obvious function.
\end{remark}

In \cite{Spivak:2013}, it is shown that packs and wiring diagrams form a category-theoretic structure, called an \emph{operad}, and that relations form an \emph{algebra} on this operad; a similar point of view is found in \cite{Fong}. While the details are beyond the scope of this paper, there are two basic and interacting ideas here. First, wiring diagrams can be nested inside each other to form a more complex wiring diagram. That is, we can zoom in and out. Second, a wiring diagram specifies an array multiplication formula, and this formula is associative with respect to nesting. We will discuss this formula in Section~\ref{sec:GAM}.

\section{Plotting relations as arrays}\label{sec:plotting}

Now that we have introduced packs and wiring diagrams, we are ready to  apply the pixel array method. The first step is to plot relations as arrays, one for each pack $(P,a,b,r)$. Let $B\coloneqq\BBox(P)\ss\Real^{\#P}$ be the associated bounding box, as in Definition~\ref{def:pack}.

Our goal is to take any functionally-defined relation on $B$ and associate to it a boolean array $A\colon\Entr(P)\to\{0,1\}$, called its \emph{plot}; see Definition~\ref{def:entries}. We use the resolutions $r(p)\in\NN$, one for each $p\in\{1,\ldots,\#P\}$, to divide $B$ into equally-sized \emph{pixels}. Each entry $e\in\Entr(P)$ in the array corresponds to one of these subcubes, i.e.\ we can define a function $B\to\Entr(P)$ under which the preimage of an entry $e\in\Entr(P)$ is the following subset of $B$:
\begin{equation}\label{eqn:Pixel_def}
\Pixel(e) = \prod_{p \in P} \Big[a_p + \delta_p*\big(\pi_p(e) - 1\big), a_p + \delta_p*\pi_p(e)\Big)
\end{equation}
where $\pi_p\colon\Entr(P)\to\{1,\ldots,r(p)\}$ returns the $p$th coordinate and $\delta_p\coloneqq\frac{b_p-a_p}{r_p}$. Each pixel is thus a half-open subcube of $B$, whose $p$th coordinate has length $\delta_p$, and it follows that the $l_\infty$-radius of the pixel is $\delta(P)\coloneqq \max_{p\in P}(\delta_p/2)$. We will let $d(x, y)$ denote distance under the $l_\infty$-metric.

Because the bounding box $B$ is compact, any continuous function $f\colon B\to\Real^k$ is in fact \emph{uniformly continuous} on $B$. That is, for any $\epsilon>0$, there exists $\delta>0$ such that $(\epsilon,\delta)$ satisfies the uniform continuity relation:
\begin{equation}\label{eqn:continuity}
\forall x,y\in B, d(x,y)<\delta\Rightarrow d(f(x),f(y))<\epsilon.
\end{equation}
For any pack $P$, let $\delta(P)$ be its radius as above. We say that $\epsilon$ is \emph{valid tolerance for $f$ on $P$} if $(\epsilon,\delta(P))$ satisfies the uniform continuity relation \eqref{eqn:continuity}.

The relations on $B$ that we want to consider plotting are those generated by systems of $k$ equations $f_i(x_1,\ldots,x_n)=0$, where $1\leq i\leq k$, or systems of inequalities $f\leq 0$ or $f<0$. The results of the pixel array method are slightly more general relations, which we call existentially-defined, as we now make precise.

\begin{definition}\label{def:functionally_defined}
Suppose given a pack $P$ with bounding box $B\coloneqq\BBox(P)$. A \emph{functionally-defined relation on $P$} is a relation of the form
\[R_{f,S}\coloneqq \{x\in B\mid f(x)\in S\},\] 
where $f\colon B\to\Real^k$ is a function for some $k\in\NN$ and $S\ss\Real^k$ is a subset. We refer to $f$ as the \emph{defining function} and to $S$ as the \emph{target subset}.

More generally, an \emph{existentially-defined relation on $P$} is a relation of the form
\[R_{E,\pi,f,S}\coloneqq \{\pi(x)\in B\mid x\in E, f(x)\in S\},\]
where $E\ss\Real^m$ is a region for some $m\in\NN$, where $\pi\colon E\to B$ and $f\colon E\to\Real^k$ are continuous functions, and where $S\ss\Real^k$ is a subset. Let $\Rel(P)$ denote the set of all existentially-defined relations on $P$.
\end{definition}
To be clear, the relations $f(x)=0$, $f(x)\leq 0$, and $f(x)<0$ would have target subsets $\{s\in\Real\mid s=0\}$, $\{s\in\Real\mid s\leq 0\}$, and $\{s\in\Real\mid s<0\}$ respectively.

\begin{remark}
It is easy to see that functionally-defined relations are a special case of existentially defined relations, taking $E=B$ and $\pi=\id_E$.
\end{remark}

Before describing a method for plotting functionally-defined relations, we first describe useful ways of bounding the error of such a plot; we will eventually see that these bounds are preserved by array multiplication. We do not assume that one can calculate values of the function $f$ perfectly, but only up to some precision, as follows.

\begin{definition}\label{def:plot_error}
    Let $R_{f,S}\ss B$ be a functionally-defined relation on a bounding box $B$. For any subset $N\ss B$, we define the \emph{$N$-error set of $R_{f,S}$} to be the set of distances
    \[D_f(N,S)\coloneqq\big\{d\big(f(x),y\big)\in\Real_{\geq 0}\mid x\in N, y\in S\big\}.%
    \footnote{If $R_{E,\pi,f,S}$ is an existentially-defined relation, then for any $N\ss B$, define the \emph{$N$-error set of $R_{E,\pi,f,S}$} as 
    \[D_f(N,S)\coloneqq\big\{d\big(f(x),y\big)\in\Real_{\geq 0}\mid \pi(x)\in N, y\in S\big\}.\]
    }
    \]
 
    For $\ell,u\geq 0$, we say that the \emph{$N$-error of $R_{f,S}$ is always above $\ell$} if $D_f(N,S)\ss(\ell,\infty)$, the \emph{$N$-error of $R_{f,S}$ achieves $u$} if $D_f(N,S)\cap[0,u]\neq\emptyset$, and the \emph{$N$-error of $R_{f,S}$ is bounded by $u$} if $D_f(\{x\},S) \cap [0,u]\neq\emptyset$ for all $x \in N$.
\end{definition}

\begin{proposition}\label{prop:contplot}
Suppose given a pack $P$ with bounding box $B\coloneqq\BBox(P)$, a functionally-defined relation $R$ on it, and a valid tolerance $\epsilon$. Choose an entry $e\in\Entr(P)$, let $\Pixel(e)\ss B$ be the corresponding pixel as in \eqref{eqn:Pixel_def}, and let $c_e$ be its center point. Then:
\begin{itemize}
  \item if the $\{c_e\}$-error of $R$ is always above $\ell$ then the $\Pixel(e)$-error is always above $\ell - \epsilon$.
  \item if the $\{c_e\}$-error of $R$ is achieves $u$ then the $\Pixel(e)$-error is bounded by $u + \epsilon$.
\end{itemize}
\end{proposition}
\begin{proof}
    Both facts follow from the fact that $d$ is a metric on $\Real^k$. For the first, given $x\in\Pixel(e)$ and $y\in S$, we have $d(x, c_e) < \delta(P)$, so $d(f(x), f(c_e)) < \epsilon$, so $\ell < d(f(c_e), y) < \epsilon + d(f(x), y)$.

For the second fact, we again have $d(f(x), f(c_e)) < \epsilon$ for any $x\in\Pixel(e)$, and there is some $y \in S$ such that $d(f(c_e), y) \leq u$. Thus we find that $d(f(x), y) < u + \epsilon$.
\end{proof}

We are now ready to describe a method for plotting relations for which we have control of the error.

\begin{definition}\label{def:sample_in_center}
Suppose given a pack $P$ with bounding box $B\coloneqq\BBox(P)$ and a functionally-defined relation $R$ on it. For each entry $e\in\Entr(P)$, let $c_e$ denote the center of the corresponding $\Pixel(e)$. For any valid tolerance $\epsilon$, we define the \emph{$\epsilon$-tolerance sample-in-center plot} to be the array $A\colon\Entr(P)\to\{0,1\}$ given by
\[
A(e)=
\begin{cases}
1&\text{ if the }\{c_e\}\text{-error achieves }\epsilon;\\
0&\text{ otherwise.}
\end{cases}
\]
\end{definition}
For example, to plot an equation $f(x_1,\ldots,x_n)=0$, we have $k=1$, $S=\{0\}$ and $R=R_{f,S}$. The $\{c_e\}$-error of $R$ achieves $\epsilon$ iff $f(c)\leq\epsilon$. Thus the sample-in-center plot reduces to simply calculating $f$ at the center point of each pixel, and determining whether or not it is less than the threshold $\epsilon$.


For an entry $e$, what does the boolean value $A(e)\in\{0,1\}$ tell us about the relationship between the points $x\in\Pixel(e)$, the function $f$, and the target subset $S$? We now define the pixel array method's accuracy guarantee.

\begin{definition}\label{def:accurate_plot}
Let $P$ be a pack, $R$ an existentially-defined relation on it, and $\epsilon>0$. We say that $A\in\Arr(P)$ is an \emph{$\epsilon$-accurate plot of $R$ on $P$} if the following hold for any entry $e\in\Entr(P)$:
\begin{compactitem}
    \item If the $\Pixel(e)$-error of $R$ achieves $0$ then $A(e) = 1$.
    \item If the $\Pixel(e)$-error of $R$ is always above $\epsilon$ then $A(e) = 0$.
    \item If $A(e) = 1$ then the $\Pixel(e)$-error of $R$ is bounded by $2\epsilon$.
\end{compactitem}
\end{definition}

The following is a corollary of Proposition~\ref{prop:contplot}.
\begin{corollary}\label{cor:guarantees}
With notation as in Definition~\ref{def:sample_in_center}, let $A$ be the $\epsilon$-tolerance sample-in-center plot. Then it is an $\epsilon$-accurate plot in the sense of Definition~\ref{def:accurate_plot}.
\end{corollary}

This guarantee also yields a potential computational optimization by not plotting any region $N\ss B$ for which the $N$-error is always above $\epsilon$.

\section{Generalized array multiplication}\label{sec:GAM}

A wiring diagram specifies an array multiplication formula, as formalized in the following theorem. 
\begin{theorem}[\cite{Spivak:2013}]\label{thm:associative}
To every wiring diagram $\Phi\colon P_1,\ldots, P_n\to P'$ we can associate an \emph{array multiplication} formula, i.e.\ a function 
\[\Arr(\Phi)\colon\Arr(P_1)\times\cdots\times\Arr(P_n)\to\Arr(P')\]
This functions is monotonic and composes associatively under nesting of wiring diagrams.
\end{theorem}
The array multiplication formula $\Arr(\Phi)$---which takes in arrays for inner packs and produces their product array in the outer pack---will be given below \eqref{eqn:GAM}. The monotonicity of $\Arr(\Phi)$ is a formalization of the "no false negatives" assertion we made in Section~\ref{ch:introduction}. That is, given arrays $A,B\in\Arr(P)$ for some pack $P$, we write $A\leq B$ if $A(e)\leq B(e)$ for all entries $e\in\Entr(P)$; monotonicity means that the function $\Arr(\Phi)$ preserves this ordering. The associativity of $\Arr(\Phi)$ is meant in the precise technical sense that $\Arr$ is an algebra on the operad of wiring diagrams; see \cite{Spivak:2013}. In simple terms, one can multiply arrays in any order and will obtain the same result.

The remaining goal for Section~\ref{sec:GAM} is to give an algorithm for $\Arr(\Phi)$ and to show that it generalizes matrix multiplication, trace, and Kronecker product. So fix a wiring diagram $\Phi\colon P_1,\ldots,P_n\to P'$, with links $\Lambda=(\ell_1,\ldots,\ell_m)$, and suppose given an array $A_i\in\Arr(P_i)$ for each $i$. We will construct the array $\Arr(\Phi)(A_1,\ldots,A_n)\in\Arr(P')$. Recall the functions $\Entr_\Phi^i\colon\Entr(\Lambda)\to\Entr(P_i)$ and $\Entr_\Phi'\colon\Entr(\Lambda)\to\Entr(P')$ from \eqref{eqn:Entr_map}.

\begin{algorithm}
  \caption{Generalized array multiplication}
    \label{alg:AMF}
  \begin{algorithmic}[1]
    \Require{Wiring diagram $\Phi\colon P_1,\ldots,P_n\to P'$ and arrays $A_i\in\Arr(P_i)$.}
    \Statex
    \Function{$\Arr(\Phi)$}{$A_1,\ldots,A_n$}
      \Let{$A'$}{$0$}  \Comment{$A'\in\Arr(P')$}
      \For{$e\in\Entr(\Lambda)$}
        \Let{$a_e$}{1}
        \For{$i\in\{1,\ldots,n\}$}
          \Let{$e_i$}{$\Entr_{\Phi}^i(e)$}   \Comment{See Equation~\eqref{eqn:Entr_map}}
          \Let{$a_e$}{$a_e*A_i(e_i)$}  \Comment{See Section~\ref{notation}}
        \EndFor
        \Let{$e'$}{$\Entr_{\Phi}'(e)$}  \Comment{See Equation~\eqref{eqn:Entr_map}}
        \Let{$A'(e')$}{$A'(e')+a_e$}  
      \EndFor
      \State \Return{$A'$}
    \EndFunction
  \end{algorithmic}
\end{algorithm}
The following \emph{generalized array multiplication formula} says the same thing and is a direct generalization of the usual matrix multiplication formula:%
\footnote{As mentioned in \ref{notation}, the operations $+$ and $*$ in $\Bool$ are given by OR ($\vee)$ and AND ($\wedge$). We use symbols $+$ and~$*$ because they more familiar in the context of matrix multiplication, and because the algorithm works in any semiring $\BB$. Similarly, in the formula \eqref{eqn:GAM}, symbols $\sum$ and $\prod$ are given by $\bigvee$ and $\bigwedge$.}
\begin{equation}\label{eqn:GAM}
A'(e')\coloneqq\sum_{\parbox{.7in}{\centering\scriptsize$e\in\Entr(\Lambda)$\\$\Entr'_\Phi(e)=e'$}}\;\;\prod_{i=1}^n A_i\left(\Entr_\Phi^i(e)\right)
\end{equation}

\begin{theorem}
The generalized array multiplication formula \eqref{eqn:GAM} is linear in each input array, and it generalizes the usual matrix multiplication, trace, and Kronecker product operations.
\end{theorem}
\begin{proof}
It is easy to see from \eqref{eqn:GAM} that array multiplication is linear in the sense that if an input array is equal to a linear combination of others, say $A_1=c*M+N$, then the result will be the respective linear combination,
\[\Arr(A_1,\ldots,A_n)=c*\Arr(M,A_2,\ldots,A_n)+\Arr(N,A_2,\ldots,A_n).\]

The multiplication, trace, and Kronecker product operations correspond respectively to the following wiring diagrams:%
\footnote{In \eqref{dia:series_parallel} we denote packs using squares---rather than circles---just to make the diagrams look nicer.}
\begin{align}\label{dia:series_parallel}
\begin{tikzpicture}[oriented WD, baseline=(Y.center), bbx=1em, bby=1ex]
	\node[bb={1}{1},bb name=$M$] (X1) {};
	\node[bb={1}{1},right =2 of X1, bb name=$N$] (X2) {};
	\node[bb={1}{1}, fit={($(X1.north west)+(-1,3)$) ($(X1.south)+(0,-3)$) ($(X2.east)+(1,0)$)}, bb name = $MN$] (Y) {};
	\draw[ar] (Y_in1') to (X1_in1);
	\draw[ar] (X1_out1) to (X2_in1);
	\draw[ar] (X2_out1) to (Y_out1');
	\draw[label] 
		node[below left=3pt and 3pt of X1_in1]{$m$}
		node[below right=3pt and 3pt of X1_out1]{$n$}
		node[below right=3pt and 3pt of X2_out1] {$p$};
\end{tikzpicture}
\qquad\qquad
\begin{tikzpicture}[oriented WD,baseline=(cod.center), bbx=1em, bby=1ex]
	\node[bb={1}{1}, bb name=$P$] (dom) {};
	\node[bb={0}{0}, fit={(dom) ($(dom.north east)+(1,4)$) ($(dom.south west)-(1,2)$)}, bb name = $\Tr(P)$] (cod) {};
	\draw[ar] let \p1=(dom.south east), \p2=(dom.south west), \n1={\y2-\bby}, \n2=\bbportlen in (dom_out1) to[in=0] (\x1+\n2,\n1) -- (\x2-\n2,\n1) to[out=180] (dom_in1);
	\draw[label] 
		node[above left=2pt and 3pt of dom_in1] {$n$}
		node[above right=2pt and 3pt of dom_out1] {$n$};
\end{tikzpicture}
\qquad\qquad
\begin{tikzpicture}[oriented WD,baseline=(Y.center), bbx=1em, bby=1ex]
	\node[bb={1}{1},bb name=$M_1$] (X1) {};
	\node[bb={1}{1},below =2 of X1, bb name=$M_2$] (X2) {};
	\node[bb={2}{2}, fit={($(X1.north west)+(-1,3)$) ($(X2.south)+(0,-3)$) ($(X2.east)+(1,0)$)}, bb name = $M_1\otimes M_2$] (Y) {};
	\draw[ar] (Y_in1') to (X1_in1);
	\draw[ar] (Y_in2') to (X2_in1);
	\draw[ar] (X1_out1) to (Y_out1');
	\draw[ar] (X2_out1) to (Y_out2');
	\draw[label] 
		node[below left=3pt and 3pt of X1_in1]{$m_1$}
		node[below right=3pt and 3pt of X1_out1]{$n_1$}
		node[below left=3pt and 3pt of X2_in1] {$m_2$}
		node[below right=3pt and 3pt of X2_out1] {$n_2$};
\end{tikzpicture}
\end{align}
Here, $M$ is $m\times n$ and $N$ is $n\times p$, as indicated, and similarly for $P,M_1,M_2$. From this point, checking that the algorithm returns the correct array in each case is straightforward, so we explain it only for the case of matrix multiplication.

The first wiring diagram $\Phi\colon P_M,P_N\to P_{MN}$ has inner packs $P_M=\{m_1,n_1\}$ and $P_N=\{n_2,p_2\}$, and outer pack $P_{MN}=\{m',p'\}$. It consists of three links $\Lambda=\{\ell_m,\ell_n,\ell_p\}$, where $\ell_m=\{m_1,m'\},\ell_n=\{n_1,n_2\},\ell_p=\{p_2,p'\}$. We want to show that $A'=\Arr(\Phi)(M,N)$, as defined in \eqref{eqn:GAM}, indeed returns the matrix product, $A'=MN$.

The set of entries for $\Lambda$ and $P_{MN}$ are
\[
\Entr(\Lambda)=\ord{m}\times\ord{n}\times\ord{p}
\qquad\text{and}\qquad
\Entr(P_{MN})=\ord{m}\times\ord{p}
\]
and, as usual, the function $\Entr_\Phi'\colon\Entr(\Lambda)\to\Entr(P_{MN})$ is the projection. Thus for any $(i,k)\in\Entr(P_{MN})$, the summation in \eqref{eqn:GAM} is over the set $(\Entr_\Phi')^{-1}(i,k)=\{(i,j,k)\mid 1\leq j\leq n\}$. Since $\Entr_\Phi^M(i,j,k)=(i,j)$ and $\Entr_\Phi^N(i,j,k)=(j,k)$, we obtain the desired matrix multiplication formula:
\[A'(i,k)=\sum_{j=1}^n M(i,j)*N(j,k).\qedhere\]
\end{proof}

\section{Accuracy of the pixel array method}
Up to this point, we have explained how to plot the equations (Section~\ref{sec:plotting}) and how to combine them to find the solution to the entire system using generalized array multiplication (Section~\ref{sec:GAM}). The algorithm as stated can be made much more efficient by clustering, as we explain in Section~\ref{sec:clustering}. For now, we want to discuss the accuracy of the method, which is unaffected by clustering.

Suppose we are given a wiring diagram $\Phi\colon P_1,\ldots, P_n\to P'$ as in Definition~\ref{def:wd} and that each pack $P_i$, is assigned an existentially-defined relation $R_i\coloneqq R_{E_i,\pi_i,f_i,S_i}\in\Rel(P_i)$ as in Definition~\ref{def:functionally_defined}. To describe the resulting system, we can take the "product" of these relations, as shown below.

Recall from \eqref{eqn:injectivity} that we have an injection $\BBox(\Lambda)\hookrightarrow\prod_i^n B_i$. By abuse of notation we write $b\in\BBox(\Lambda)$ if $b$ is in the image of this function. Define
\[E'\coloneqq\left\{(x_1,\ldots,x_n)\in\prod_{i=1}^n E_i\;\middle|\;\big(\pi_1(x_1)\times\cdots\times\pi_n(x_n)\big)\in\BBox(\Lambda)\right\}.\]
Let $k'\coloneqq k_1+\ldots+k_n$, and define $S'\coloneqq\prod_i S_i\ss\Real^{k'}$. Define $f'\colon E'\to\Real^{k'}$ coordinate-wise, i.e.
\[f'(x_1,\ldots,x_n)\coloneqq \big(f_1(x_1),\ldots,f_n(x_n)\big).\]
Finally, let $\pi'\colon E'\to B'$ denote the composite of the obvious map $E'\to\BBox(\Lambda)$ and the map $\BBox(\Lambda)\to B'$ from \eqref{eqn:BBox_map}, and let $\Rel(\Phi)(R_1,\ldots,R_n)\coloneqq R_{E',\pi',f',S'}$ be the resulting existentially-defined relation.

If $A_i\in\Arr(P_i)$ is the plot of relation $R_{E_i,\pi_i,f_i,S_i}$, then multiplying these arrays according to $\Phi$, the result $A'\coloneqq\Arr(\Phi)(A_1,\ldots,A_n)$ is supposed to provide a plot of $R'\coloneqq\Rel(\Phi)(R_1,\ldots,R_n)$. Recall the notion of $\epsilon$-accurate plot from Definition~\ref{def:accurate_plot}. The following result  says that $\epsilon$-accuracy is preserved under application of the pixel array method.

\begin{theorem}\label{thm:extended}
Let $\Phi\colon P_1,\ldots,P_n\to P'$ be a wiring diagram, and suppose given $R_i\in\Rel(P_i)$ and $A_i\in\Arr(P_i)$ for each $i$. Let $R'\coloneqq\Rel(\Phi)(R_1,\ldots,R_n)$ and $A'\coloneqq\Arr(\Phi)(A_1,\ldots,A_n)$. For any $\epsilon>0$, if each $A_i$ is an $\epsilon$-accurate plot of $R_i$ on $P_i$ then $A'$ is an $\epsilon$-accurate plot of $R'$ on $P'$.
\end{theorem}
\begin{proof}
    Suppose each $A_i$ is an $\epsilon$-accurate plot of $R_i$ on $P_i$. Now consider any entry $e' \in \Entr(P')$; we will show that the three conditions of Definition~\ref{def:accurate_plot} are satisfied. These will all be in terms of the $\Pixel(e')$-error set
    \[D_{f'}\big(\Pixel(e'),S'\big)\coloneqq\left\{d\big(f'(x),y\big)\in\Real_{\geq 0}\mid \pi'(x)\in\Pixel(e'), y\in S'\right\}.\]
    from Definition~\ref{def:plot_error}, where $f'$, $S'$, and $\pi'$ are as in the paragraphs above.

    If the $\Pixel(e')$-error of $R'$ achieves $0$, then there is some $x=(x_1,\ldots,x_n) \in E'$ such that $\pi'(x)\in\Pixel(e')$ and $f'(x) \in S'$. Let $i\in\{1,\ldots,n\}$; by definition of $E'$ there is a unique entry $e_i\in\Entr(P_i)$ such that $\pi_i(x_i)\in\Pixel(e_i)$. By definition of $S'$, we have that $f_i(x_i) \in S_i$, so  the $\Pixel(e_i)$-error of $R_i$ achieves $0$, and thus by assumption $A_i(e_i) = 1$. Their product $\prod_i A_i=1$ is a summand of Equation~\eqref{eqn:GAM}, so $A'(e') = 1$.

    Next, suppose the $\Pixel(e')$-error is always above $\epsilon$. For any $x\in E'$ such that $\pi'(x)=e'$, let $e_1,\ldots,e_n$ be as above. Again by definition of the $l_\infty$-metric, there is some $i$ such that the $\Pixel(e_i)$-error is always above $\epsilon$, so by assumption $A_i(e_i) = 0$. This contributes a $0$ to every product term $\prod_{i}A_i(e_i)$ of \eqref{eqn:GAM}, so their sum is again $A'(e') = 0$.


    Finally, if $A'(e') = 1$ then by \eqref{eqn:GAM} there is an $e\in\Entr(\Lambda)$ such that $\Entr'_\Phi(e)=e'$ and $A_i(e_i) = 1$ for all $i$, where $e_i=\Entr_\Phi^i(e)$. By assumption, the $\Pixel(e_i)$-error of $R_i$ is bounded by $2\epsilon$ so, according to the $l_\infty$-metric, the $\prod_i \Pixel(e_i)$-error of the product relation also is bounded by $2\epsilon$. But $\Pixel(e')$ is just a projection of this product, so the $\Pixel(e')$-error of $R'$ also is bounded by $2\epsilon$.
\end{proof}

Note that Corollary~\ref{cor:guarantees} is the base case of the property in Theorem~\ref{thm:extended}, where there is one function and target set, and no unexposed variables. Thus, Theorem~\ref{thm:extended} shows that the array multiplication guarantees, even through arbitrary levels of nesting, a solution plot to a system of relations (represented here as a product) that has the same accuracy guarantees as the plots of atomic relations. These initial plots can be generated in a straightforward way as shown in Definition~\ref{def:sample_in_center}.

\section{Clustering to minimize the cost polynomial}\label{sec:clustering}

One can readily see that the cost---i.e.\ computational complexity---of naively performing the generalized array multiplication algorithm~\ref{alg:AMF} on a wiring diagram $\Phi$ with links $\Lambda=\{1,\ldots,n\}$, having resolutions $r_1,\ldots,r_n\geq 2$, will be the product $r_1*\cdots*r_n$, since the algorithm iterates through all entries $e\in\Entr(\Lambda)$.%
\footnote{Throughout this section, we refer to the \emph{naive cost} of matrix/array multiplication. Improving this using modern high-performance methods, or by taking advantage of sparsity, is left for future work.}
For any set $L$ and function $r\colon L\to\NN_{\geq2}$, we denote the product by
\begin{equation}\label{eqn:naive_cost}
	r^L\coloneqq\prod_{\ell\in L}r(\ell).
\end{equation}
It is often convenient to assume that the resolution is \emph{constant}, meaning that there is a fixed $r_0\geq 2$ such that $r(\ell)=r_0$ for all $\ell\in L$. In this case our notation agrees with the usual arithmetic notation, $r^L=r_0^{\#L}$. Equation \eqref{eqn:naive_cost} gives an upper bound for the computational complexity, so we already see that the complexity of the array multiplication formula for $\Phi$ is at most polynomial in $r$.

In fact, we can reduce the degree of this polynomial by \emph{clustering} the diagram. The savings is related to the number (and resolution) of links that are properly contained inside clusters, i.e.\ that represent unexposed variables. Clustering in this way returns the correct answer exactly because the generalized array multiplication formula is associative in the sense of Theorem~\ref{thm:associative}. For example, to multiply three $n\times n$-matrices $MNP$, the unclustered cost would be $n^4$, whereas the clustered cost---obtained by using the associative law---is $2n^3$. We now discuss what we mean by clustering for a general wiring diagram.

\begin{definition}\label{def:efficient_factorization}
Let $\Phi\colon P_1,\ldots,P_n\to P'$ be a wiring diagram. A \emph{cluster} is a choice of subset $C\ss\{1,\ldots,n\}$; we may assume by symmetry that $C=\{1,\ldots,m\}$ for some $m\leq n$. 

Let $\varphi\colon P_1\sqcup\cdots\sqcup P_n\sqcup P'\to \Lambda$ be the partition as in Definition~\ref{def:wd}. Consider the images
\[
\Lambda'_C\coloneqq\varphi(P_1\sqcup\cdots\sqcup P_m)
\qquad\text{and}\qquad
\Lambda''_C\coloneqq\varphi(P_{m+1}\sqcup\cdots\sqcup P_n\sqcup P')
\]
which we call the sets of \emph{$C$-interior links} and \emph{$C$-exterior links}, respectivly. Let $Q_C=\Lambda'_C\cap \Lambda''_C$ be their intersection; we call $Q_C$ the \emph{$C$-intermediate pack}. Define 
\[
\Phi'_C\colon P_1,\ldots,P_m\to Q_C
\qquad\text{and}\qquad
\Phi''_C\colon Q_C,P_{m+1},\ldots,P_n\to P'
\]
to be the evident restrictions of $\varphi$ with links $\Lambda'_C$ and $\Lambda''_C$, respectively. We call $\Phi'_C$ the \emph{interior diagram} and $\Phi''_C$ the \emph{exterior diagram}, and refer to the pair $(\Phi'_C,\Phi''_C)$ as the \emph{$C$-factorization of $\Phi$}.%
\footnote{
It is easy to show that, in the language of operads, $\Phi=\Phi''\circ_C\Phi'$ is indeed a factorization; see \cite{Leinster}.
}
We may drop the $C$'s when they are clear from context.

\end{definition}

Clustering is worthwhile if it separates internal from external variables in a sense formalized by the following definition and proposition.

\begin{definition}\label{def:properly_internal_external}
With notation as in Definition~\ref{def:efficient_factorization}, we say that $C$ is a \emph{trivial cluster} if either $\Lambda'_C=Q_C$ or $\Lambda''_C=Q_C$. We refer to  $L'\coloneqq \Lambda'_C-Q_C$ (resp.\ $L''=\Lambda''_C-Q_C$) as the set of \emph{properly internal \textnormal{(resp. \emph{properly external})} links} in the $C$-factorization. Thus $C$ is trivial iff either $L'=\emptyset$ or $L''=\emptyset$; otherwise we say that $C$ is \emph{a nontrivial cluster}.
\end{definition}

\begin{proposition}\label{prop:speedup}
Let $\Phi\colon P_1,\ldots,P_n\to P'$ be a wiring diagram, and let $C\ss\ord{n}$ be a nontrivial cluster. Let $Q$ be its intermediate pack, and let $L',L''\neq\emptyset$ be the sets of properly internal and external links in the $C$-factorization (Definition~\ref{def:properly_internal_external}). Then clustering at $C$ is always "efficient", in the sense that the cost difference is nonnegative:
\begin{equation}\label{eqn:difference}
r^{Q}\Big(r^{L'+L''}-\big(r^{L'}+r^{L''}\big)\Big)\geq 0.
\end{equation}
\end{proposition}

\begin{proof}
By assumption, and with notation as in \eqref{eqn:naive_cost}, we have $r^{L'},r^{L''}\geq 2$. Noting that the intermediate pack $P$ may be empty, we have $r^Q\geq 1$; thus the inequality holds, and it will be strict if any of these three quantities are not at their lower bound. The formula in \eqref{eqn:difference} is a simple matter of applying the comparing the cost of $\Phi$, given in \eqref{eqn:naive_cost}, to the sum of the costs for the clusters, $\Phi'$ and $\Phi''$. One uses that $\Lambda=Q+L+L''$, $\Lambda'=L'+Q$, and $\Lambda''=L''+Q$.
\end{proof} 

\subsection{Cluster trees}
Given a wiring diagram $\Phi\colon P_1,\ldots,P_n\to P'$, one may perform a sequence of clusterings, either in parallel or in series. If two clusterings can be done in parallel, we do not differentiate between the order in which they are performed. We will refer to the resulting dependency tree as the \emph{cluster tree} $T$; it is also called a \emph{non-binary, labeled, non-ranked dendogram} (see \cite[Section 7]{MURTAGH1984191}). Its leaves represent the inner packs of $\Phi$, its branches represent intermediate packs, and its nodes represent (possibly trivial) clusters. We label each node in $T$ by the cost \eqref{eqn:naive_cost} of the generalized array multiplication algorithm for the associated cluster. Let $\Clust(\Phi)$ denote the set of all cluster trees for $\Phi$. The costs in each node of a tree $T$ can be summed resulting in a polynomial $\Cost_T\in\NN[r]$%
\footnote{
Here we assume that the resolution $r\colon\Lambda\to\NN_{\geq 2}$ is constant, so $\NN[r]$ denotes the polynomial semiring in one variable with coefficients in $\NN$. Note that $\NN[r]$ has the structure of a linear order, where for example $r^2<r^2+r<2r^2<r^3$. In fact $\NN[r]$ is a well-order: every subset of elements has a minimum. It is easy to generalize to a non-constant resolution function $r$ by requiring a variable for each $\ell\in\Lambda$; however the linear ordering is lost in so doing. 
}
called the \emph{cost polynomial for $T$}, which represents the total serial runtime associated to this clustering. Given access to unlimited parallel computational resources, the cost would be given by a polynomial $\PCost_T$, which we call the \emph{parallel cost polynomial}. It is defined recursively, using max and sum: add the cost at a node to the max cost of its descendents \cite{jaja1992}.

There are many cluster trees $T$ for a wiring diagram $\Phi$, and each has its own (parallel) cost polynomial. Assuming constant resolution $r\geq2$, we may take the minimum over all (parallel) cost polynomials. We denote the result by 
\[
\Cost_\Phi\coloneqq\min_{T\in\Clust(\Phi)}\Cost_T
\qquad\text{and\ }\qquad
\PCost_\Phi\coloneqq\min_{T\in\Clust(\Phi)}\PCost_T
\]
and refer to it as the \emph{cluster polynomial} (resp.\ the \emph{parallel cluster polynomial}) for $\Phi$. 

\begin{example}
Here we show a wiring diagram (solid lines) together with a cluster tree $T$:
\begin{equation}
\begin{tikzpicture}[baseline=(outer.center)]
	\node[pack, scale = .75] at ( 0.0,  0.0) (D) {D};
	\node[pack, scale = .75] at ( 1.1,  0.0) (G) {G};
	\node[pack, scale = .75] at ( 2.4,  0.5) (E) {E};
	\node[pack, scale = .75] at ( 2.4, -0.5) (F) {F};
	\node[pack, scale = .75] at (-2.4,  0.0) (C) {C};
	\node[pack, scale = .75] at (-3.4,  0.5) (A) {A};
	\node[pack, scale = .75] at (-3.4, -0.5) (B) {B};
	
	\node[helper] at (-5.2,  0.5) (h1) {};
	\node[helper] at (-5.2, -0.5) (h2) {};
	\node[helper] at (-2.9,  0.0) (h3) {};
	\node[helper] at ( 5.0,  0.5) (h4) {};
	
	\draw (A) -- +({-1*cos(30)},{1*sin(30)});
	\draw (B) -- +({-1*cos(30)},{-1*sin(30)});
	\draw (h3) to (A);
	\draw (h3) to (B);
	\draw (h3) to (C);
	\draw (C) [bend right = 30] to (D);
	\draw (C) [bend left = 30] to (D);
	\draw (D) to (G);
	\draw (G) to (F);
	\draw (E) [bend right = 30] to (F);
	\draw (E) [bend left = 30] to (F);
	\draw (E) -- +(2,0);
	
	\node[invertwd, dashed, fit={(E) (F)}, scale = .21] {};
	\node[pack, dashed, fit = {(A) (B) (C)}, scale = .9] (p) {};
	\node[pack, dashed, fit = {(G) (E) (F)}, scale = .90] {};	
	\node[draw, ellipse, dashed, fit = {(G) (E) (F) (D)}, scale = 1.05] (q) {};
	\node[draw, ellipse, fit = {(p) (q)}, scale = .83] (outer) {};	
\end{tikzpicture}
\end{equation}
It could also be denoted $T=\{\{A,B,C\},\{\{\{E,F\},G\},D\}\}$ or drawn as follows:
\[
\begin{tikzpicture}[oriented WD,bb port sep=1, bb port length=2.5pt, bbx=.6cm, bb min width=.4cm, bby=.7ex]
	\node[bb={3}{1}] (abcz) {$r^5$};
	\node[bb={2}{1}, below=3 of abcz] (efy) {$r^4$};
	\node[bb={2}{1}, below right=-2 and 1 of efy] (ygx) {$r^3$};
	\node[bb={2}{1}, above right =1 and 1 of ygx] (ZDV) {$r^4$};
	\node[bb={2}{1}, above right =-2 and 1 of ZDV] (vxw) {$r^5$};
	\node[bb={7}{1}, fit={(abcz) (ZDV) (efy) (ygx) (vxw)}, dotted] (outer) {};
	\draw (outer_in1') to (abcz_in1);
	\draw (outer_in2') to (abcz_in2);
	\draw (outer_in3') to (abcz_in3);
	\draw let \p1=(abcz.west), \p2=(ZDV_in1), \n1=\bbportlen in
		(outer_in4') to (\x1-\n1,\y2) -- (ZDV_in1);
	\draw (outer_in5') to (efy_in1);
	\draw (outer_in6') to (efy_in2);
	\draw let \p1=(efy.west), \p2=(ygx_in2), \n1=\bbportlen in
		(outer_in7') to (\x1-\n1,\y2) -- (ygx_in2);
	\draw let \p1=(ZDV.east), \p2=(abcz_out1), \n1=\bbportlen in
		(abcz_out1) -- (\x1+\n1,\y2) to (vxw_in1);
	\draw (efy_out1) to (ygx_in1);
	\draw (ZDV_out1) to (vxw_in2);
	\draw (ygx_out1) to (ZDV_in2);
	\draw (vxw_out1) to (outer_out1');
	\draw[label] 
		node[left=2pt of outer_in1]{$A$}
		node[left=2pt of outer_in2]{$B$}
		node[left=2pt of outer_in3]{$C$}
		node[left=2pt of outer_in4]{$D$}
		node[left=2pt of outer_in5]{$E$}
		node[left=2pt of outer_in6]{$F$}
		node[left=2pt of outer_in7]{$G$}
	%
	;
\end{tikzpicture}
\]
The total cost (computational complexity) of array multiplication for this clustering is the sum $\Cost_T=2r^5+2r^4+r^3$. It is not hard to show that this is the minimal cost; i.e.\ that the cluster polynomial for $\Phi$ is $\Cost_\Phi=\Cost_T$. Similarly, $\PCost_\Phi=\PCost_T=2r^5$.
\end{example}

\subsection{The value of clustering}\label{sec: success of algos}
To determine the savings based on various clustering algorithms, we generated a number of random packs and wiring diagrams. We wrote a script that generated 1000 random wiring diagrams with $1\leq n\leq 10$ inner packs, and determined which of our clustering techniques was most efficient. We then recorded what fraction of the unclustered cost it achieved, and calculated the mean over all trials. The results are shown below; one can see that the efficiency of clustering is roughly exponential in the number of packs.

\[
\begin{tikzpicture}
\begin{semilogyaxis}[
title = \footnotesize Number of Packs vs. Proportion of Max Time,
xlabel = \footnotesize Number of Inner Packs,
ylabel = \footnotesize Clustering time / max time,
xmin = 1, xmax = 9,
width=2.5in
]
\addplot[black, mark=o] coordinates {
	(1, 1)
	(2, 1)
	(3, .0140707)
	(4, .00177485)
	(5, .000129325)
	(6, 4.49858e-6)
	(7, 3.02081e-7)
	(8, 5.11764e-8)
	(9, 1.96023e-9)
};
\end{semilogyaxis}
\end{tikzpicture}
\]

\section{Mathematical summary of pixel array method}
Recall from Section~\ref{sec:input} that the input to the pixel array method is a set of relations $R_1,\ldots,R_n$, a discretization---range and resolution---for each variable, and a choice of variables to expose. To each relation $R_i$ we associate a pack $P_i$ whose ports are the variables that occur in $R_i$. These are the inner packs of a wiring diagram $\Phi$ whose outer pack $P'$ is the set of exposed variables, and whose links represent shared variables; see Section~\ref{sec:Packs_and_WDs}.

A plot $A_i\in\Arr(P_i)$ for each relation $R_i$ is created using any sufficiently fast and accurate method. These plots are then combined using generalized array multiplication $\Arr(\Phi)$, as specified by the wiring diagram $\Phi$; see Section~\ref{sec:GAM}. The result is an array $A'\in\Arr(P')$, namely the plot of solutions of the whole system, in terms only of the exposed variables. Error bounds on the initial plots are preserved under this array multiplication process, as shown in Theorem~\ref{thm:extended}. By associativity, the array multiplication can be clustered without affecting the solution, thus speeding up the computation considerably; see Section~\ref{sec:clustering}. By monotonicity, array multiplication will not introduce false negatives, and the result is an approximate solution set to the whole system.

\chapter{Results}\label{ch:Results}
In this section, we discuss our attempt to benchmark the pixel array method. We also discuss potential future work.


We made several attempts to benchmark the efficiency of our implementation against other well-known nonlinear solvers. However, we ran into difficulty because the question we are asking is fundamentally different from the one solved by other methods. Namely, we are interested in finding all solutions in a given box rather than finding a single solution given an initial guess. Thus in order to make the comparison, we attempted to apply other methods to solve our problem. We compared the speed of the PA method with that of two other solvers, namely Julia's NLsolve (\url{https://github.com/EconForge/NLsolve.jl}) and Mathematica's NSolve (\url{https://reference.wolfram.com/language/ref/NSolve.html}), believing them to be roughly representative of the state-of-the-art.

To test against NLsolve, we wrote a script to generate a discretized array that pixelates the $n$-dimensional hypercube of all possible solutions to a given system. We then iterated over each pixel in that array, running NLsolve with the pixel's center point as the initial guess and with the pixel radius as the precision. If NLsolve returns a solution that lands anywhere in our bounding box, we turn on the corresponding pixel. While in some sense crude, this procedure simulates finding all solutions in our bounding box.

We timed the PA method against the above procedure using NLsolve for a few different systems. For example, we used the system presented in Example~\ref{ex: butterfly}, where each variable has a resolution of 50, and variables $x$ and $z$ are exposed. In this case, NLsolve did not terminate after running for over three hours on a Dell laptop, while the PA method took about 1.5 seconds,
a speedup of more than 7,200x. We did not do an exhaustive study, but in every case we tried, the PA method was much faster than NLsolve.

We attempted to solve the same system (Example~\ref{ex: butterfly}) using Mathematica's methods for solving systems of equations. The usual Solve method fails because our system probably has no closed-form solution. On systems with a closed-form solution, Solve is superior to our method; however, such systems are rare in practice. We tried using NSolve to find numerical solutions to this system, but it threw an exception, "This system cannot be solved with the methods available to NSolve." Seeing Mathematica's Solve and NSolve completely fail illustrates a major advantage of our method: since it solves systems simply by pixelating their graphs in a bounding box, the pixel array method is virtually unaffected by the form of the functions in the system.

\subsection{Future work}



The speed of the pixel array method could be improved with help from the matrix arithmetic community. It would be useful to have fast algorithms for general array multiplication, as described in Section~\ref{sec:GAM}. For example, given three (possibly sparse) arrays that share one or more dimensions, there must surely be faster techniques for multiplying them together than the naive algorithm we supply above. It also seems likely that a clustering algorithm for pointed hypergraphs---something like the one given in \cite{Klimmek96asimple}---could be useful. 

\chapter*{Acknowledgments}\label{ch: acknowledgments}

The authors thank Andreas Noack from the Julia computing group at MIT, who was very generous with his time when answering our questions about Julia. We also thank Patrick Schultz and Spencer Breiner for their helpful comments on a draft of this paper.

\appendix
\chapter{Sample code}\label{sec:sample_code}
Below is sample Julia code, which plots the equations from Figure~\ref{fig: motivating example} as matrices and multiplies them.
\scriptsize
\begin{verbatim}
immutable Dim
  lower::Float64
  upper::Float64
  resolution::Int
end

function make_mat(f::Function, xdim::Dim, ydim::Dim, tol::Float64)
  rows = xdim.resolution
  cols = ydim.resolution
  result = Matrix{Bool}(rows, cols)

  xstep = (xdim.upper - xdim.lower) / xdim.resolution
  ystep = (ydim.upper - ydim.lower) / ydim.resolution

  for i in 1 : rows
    for j in 1 : cols
      xval = xdim.lower + (i - 0.5)*xstep
      yval = ydim.lower + (j - 0.5)*ystep          
      result[i, j] = abs( f(xval, yval) ) < tol
    end
  end

  return result
end

dim = Dim(-1.2, 1.2, 50)

M1 = make_mat((x,y) -> y - x^2, dim, dim, .05)
M2 = make_mat((y,z) -> y + z^2 - 1, dim, dim, .05)

using UnicodePlots
spy(M1)
spy(M2)
spy(M1*M2)
\end{verbatim}
\normalsize

\printbibliography

\end{document}